\documentclass[oneside,english]{amsart}
\usepackage[T1]{fontenc}
\usepackage[latin9]{inputenc}
\usepackage{amsthm}
\usepackage{amstext}
\usepackage{amssymb}
\usepackage{esint}
\usepackage{graphicx}
\usepackage{enumerate}
\usepackage{amscd}

\makeatletter
\newtheorem{theorem}{Theorem}[section]
\newtheorem{lemma}[theorem]{Lemma}

\numberwithin{figure}{section}
\theoremstyle{definition}
\newtheorem{definition}[theorem]{Definition}

\theoremstyle{remark}

\numberwithin{equation}{section}
\makeatother

\usepackage{color}

\newcommand{\Om} {\Omega}

\usepackage{babel}

\newcommand\R{\mathbb R}

\begin{document}

\title[Mixed local and nonlocal Dirichlet $(p,q)$-eigenvalue problem]
{Mixed local and nonlocal Dirichlet $(p,q)$-eigenvalue problem}

\author{Prashanta Garain and Alexander Ukhlov}

\begin{abstract}
In this article, we consider the spectral problem for the mixed local and nonlocal $p$-Laplace operator. We discuss the existence and regularity of eigenfunction of the associated Dirichlet $(p,q)$-eigenvalue problem in a bounded domain $\Omega\subset\mathbb{R}^N$ under the assumption that $1<p<\infty$ and $1<q<p^{*}$ where $p^{*}=\frac{N p}{N-p}$ if $1<p<N$ and $p^{*}=\infty$ if $p\geq N$.
\end{abstract}
\maketitle

\section{Introduction }
In this article, we study the following mixed local and nonlocal $(p,q)$-eigenvalue problem with the Dirichlet boundary condition:
\begin{equation}\label{meqn}
-\Delta_p u+(-\Delta_p)^s u=\lambda\|u\|_{L^q(\Om)}^{p-q}|u|^{q-2}u\text{ in }\Om,\quad u=0\text{ in }\mathbb{R}^N\setminus\Om,
\end{equation}
where $\Om\subset\mathbb{R}^N$ is a bounded domain. We assume that $1<p<\infty$, $1<q<p^*$, where $p^*=\frac{N p}{N-p}$ if $1<p<N$ and $p^*=\infty$ if $p\geq N$. Here
$
\Delta_{p} u=\text{div}(|\nabla u|^{p-2}\nabla u)
$
is the $p$-Laplace operator and for $0<s<1$  
$$
(-\Delta_p)^s u=\text{P.V.}\int_{\mathbb{R}^N}\frac{|u(x)-u(y)|^{p-2}(u(x)-u(y))}{|x-y|^{N+ps}}\,dy,
$$
is the fractional $p$-Laplace operator, where P.V. denotes the principal value, see \cite{DPV} for more details.

Eigenvalue problems have received considerable attention over the past decades due to its strong relation with bifurcation theory, resonance problems, spectral optimization problems and also with applied sciences, such as fluid and quantum mechanics, see \cite{KLin1, Anle, Lind92, RSD}, and the references therein for further details.

In the local case, the $p$-Laplace eigenvalue problem
$$
-\Delta_p u=\lambda\|u\|^{p-q}_{L^q(\Om)}|u|^{q-2}u\text{ in }\Om,\quad u=0\text{ on }\partial\Om,
$$
has been intensively studied. In this concern, for $p=q$, we refer to the works \cite{GP,Lind92} and the references therein. The case $p\neq q$, was considered, for example, in \cite{Drabek99,DKN,ET1,Ercole,FL,GP,IO,K2,KLin1,Naz1,Otani1,Otani2}.

In the nonlocal case, the following fractional $p$-Laplace eigenvalue problem
$$
(-\Delta_p)^s u=\lambda|u|^{p-2}u\text{ in }\Om,\quad u=0\text{ in }\mathbb{R}^N\setminus\Om.
$$
was studied in \cite{FP14,LL14}. We also refer to \cite{BPari,DRSS} and the references therein.

In the mixed local and nonlocal case, in \cite{DFR,GS} the following eigenvalue problem has been addressed
\begin{equation*}
\begin{split}
-\Delta_p u-\int_{\mathbb{R}^N}
\mathcal{J}(x-y)|u(x)-u(y)|^{p-2}(u(y)-u(x))dy &=\lambda |u|^{p-2}u\text{ in }\Omega,\\
u &= 0 \text{ in } \mathbb{R}^N\setminus\Omega,
\end{split}
\end{equation*}
where $\mathcal{J}:\mathbb{R}^N\to\mathbb{R}^{+}$ is a nonsingular, radially symmetric, nonnegative and compactly supported
kernel. Further, authors  in \cite{BDD, DS} studied the limiting problem for mixed local and nonlocal problems.

Although, as far as we are aware, $(p,q)$-eigenvalue problems involving the mixed local and
nonlocal operator have been far less studied in the literature. In this article, we are mainly interested in the problem \eqref{meqn}. To be more precise, we investigate the existence and regularity properties of eigenfunctions of \eqref{meqn}. To this end, we follow the approach introduced in \cite{Ercole}.

The organization of the paper is as follows: In Section 2, we present the functional setting, state some auxiliary results and the main results of this article. In Section 3, some preliminary results are proved. Finally, in Section 4, we prove the main results.

\section{Functional Setting and main results}
In this section, we present some known results for the Sobolev spaces, see \cite{DPV} for more details.
Let 
$E\subset \mathbb{R}^N$
be a measurable set and $|E|$ denote its Lebesgue measure. Recall that the Lebesgue space 
$L^{p}(E),1<p<\infty$, 
is defined as the space of $p$-integrable functions $u:E\to\mathbb{R}$ with the norm
$$ \|u\|_{L^p(E)}=
\left(\int_{E}|u(x)|^p~dx\right)^{\frac1p}.
$$
Here and in the rest of the paper, let $\Om\subset\mathbb{R}^N$ with $N\geq 2$ be a bounded domain. 

The Sobolev space $W^{1,p}(\Omega)$, $1<p<\infty$, (see, for example, \cite{M}) is defined 
as the Banach space of locally integrable weakly differentiable functions
$u:\Omega\to\mathbb{R}$ equipped with the norm
\[
\|u\|_{W^{1,p}(\Omega)}=\| u\|_{L^p(\Omega)}+\|\nabla u\|_{L^p(\Omega)}.
\]
The space $W^{1,p}(\mathbb{R}^N)$ is defined analogously.
The Sobolev space $W^{1,p}_0(\Omega)$ is defined as the closure of $C_c^{\infty}(\Omega)$ with respect to the norm  $\|u\|=\|\nabla u\|_{L^p(\Om)}$.

The fractional Sobolev space $W^{s,p}(\Omega)$, $0<s<1<p<\infty$, is defined by
$$
W^{s,p}(\Omega)=\Big\{u\in L^p(\Omega):\frac{|u(x)-u(y)|}{|x-y|^{\frac{N}{p}+s}}\in L^p(\Omega\times \Omega)\Big\}
$$
and endowed with the norm
$$
\|u\|_{W^{s,p}(\Omega)}=\left(\int_{\Omega}|u(x)|^p\,dx+\int_{\Omega}\int_{\Omega}\frac{|u(x)-u(y)|^p}{|x-y|^{N+ps}}\,dx\,dy\right)^\frac{1}{p}.
$$

To study mixed problems, we consider the space $X$ defined as
\[
X=\{u\in W^{1,p}(\mathbb{R}^N):u|_{\Omega}\in W_0^{1,p}(\Omega), u=0\text{ a.e. in }\mathbb{R}^N\setminus\Om\}
\]
under the norm 
$$
\|u\|=\|\nabla u\|_{L^p(\Om)}+\left\|\frac{u(x)-u(y)}{|x-y|^\frac{N+ps}{p}}\right\|_{L^p(\mathbb{R}^N\times\mathbb{R}^N)}.
$$

Next, we have the following result from \cite[Lemma $2.1$]{BDD}.
\begin{lemma}\label{locnon1}
There exists a constant $c=c(N,p,s,\Omega)$ such that
\begin{equation}\label{locnonsem}
\int_{\mathbb{R}^N}\int_{\mathbb{R}^N}\frac{|u(x)-u(y)|^p}{|x-y|^{N+ps}}\,dx\,dy\leq c\int_{\Omega}|\nabla u|^p\,dx
\end{equation}
for every $u\in X$.
\end{lemma}

For the following result, see \cite{BDVV2, Biagi1, SV}.
\begin{lemma}
\label{Xuthm}
The space $X$ is a real  separable and uniformly convex Banach space.
\end{lemma}

Next, we state the following result, which follows from \cite[Theorem $9.14$]{var}.
\begin{theorem}\label{MB}
Let $V$ be a real separable reflexive Banach space and $V^*$ be the dual of $V$. Assume that $A:V\to V^{*}$ is a bounded, continuous, coercive and monotone operator. Then $A$ is surjective, i.e., given any $f\in V^{*}$, there exists $u\in V$ such that $A(u)=f$. If $A$ is strictly monotone, then $A$ is also injective. 
\end{theorem}

The next result follows from \cite[Corollary $1.57$]{Maly} (see, also, \cite{M}).
\begin{lemma}\label{embd}
Let $\Omega\subset\mathbb{R}^N$ be such that $|\Omega|<\infty$ and $1<p<\infty$, $1<r<p^{*}$. Then for every $u\in W_0^{1,p}(\Omega)$, there exists a positive constant $C=C(r,p,N)$ such that
\begin{equation}\label{embeqn}
\left(\int_{\Omega}|u|^r\,dx\right)^\frac{1}{r}\leq C|\Omega|^{\frac{1}{r}-\frac{1}{p}+\frac{1}{N}}\left(\int_{\Omega}|\nabla u|^p\,dx\right)^\frac{1}{p}.
\end{equation}
\end{lemma}

Next we define the notion of solution of the problem \eqref{meqn}.
\begin{definition}\label{def}
We say that $(\lambda,u)\in \mathbb{R}\times X\setminus\{0\}$ is an eigenpair of \eqref{meqn} if for every $\phi\in X$, we have
\begin{equation}\label{mwksol}
\begin{split}
&\int_{\Omega}|\nabla u|^{p-2}\nabla u\nabla\phi\,dx+\int_{\R^N}\int_{\mathbb{R}^N}\frac{|u(x)-u(y)|^{p-2}(u(x)-u(y))(\phi(x)-\phi(y))}{|x-y|^{N+ps}}\,dx dy\\
&=\lambda\|u\|_{L^q(\Om)}^{p-q}\int_{\Omega}|u|^{q-2}u\phi\,dx.
\end{split}
\end{equation}
\end{definition}
We observe that Lemma \ref{locnon1} ensures the above Definition in \eqref{mwksol} is well stated. We refer to $\lambda$ as an eigenvalue and $u$ as an eigenfunction of \eqref{meqn} corresponding to the eigenvalue $\lambda$.

\subsection{Main results}
Our main results in this article reads as follows:

\begin{theorem}\label{newthm}
Let $0<s<1$, $1<p<\infty$ and $1<q<p^{*}$. Then the following properties hold:
\vskip 0.2cm
\noindent
$(a)$ There exists a sequence $\{w_n\}_{n\in\mathbb{N}}\subset X\cap L^q(\Omega)$ such that $\|w_n\|_{L^q(\Omega)}=1$ and for every $v\in X$, we have 
\begin{multline}\label{its}
\int_{\Om}|\nabla w_n|^{p-2}\nabla w_n\nabla v\,dx\\
+\int_{\R^N}\int_{\R^N}\frac{|w_n(x)-w_n(y)|^{p-2}(w_n(x)-w_n(y))(v(x)-v(y))}{|x-y|^{N+ps}}\,dx dy\\
=\mu_n\int_{\Omega}|w_{n}|^{q-2}w_{n}v\,dx,
\end{multline}
where
\begin{multline*}\label{subopmin}
\mu_n\geq \lambda
:= \\
\inf\left\{\int_{\Om}|\nabla u|^p\,dx+\int_{\R^N}\int_{\R^N}\frac{|u(x)-u(y)|^{p}}{|x-y|^{N+ps}}\,dx dy:u\in X\cap {L^q(\Omega)},\,\|u\|_{L^q(\Omega)}=1\right\}.
\end{multline*}
\vskip 0.2cm
\noindent
(b) Moreover, the sequences $\{\mu_n\}_{n\in\mathbb{N}}$ and $\{\|w_{n+1}\|_{X}^{p}\}_{n\in\mathbb{N}}$ given by \eqref{its} are nonincreasing and converge to the same limit $\mu$, which is bounded below by $\lambda$. Further, there exists a subsequence $\{n_j\}_{j\in\mathbb{N}}$ such that both $\{w_{n_j}\}_{j\in\mathbb{N}}$ and $\{w_{n_{j+1}}\}_{j\in\mathbb{N}}$ converges in $X$ to the same limit $w\in X\cap L^q(\Omega)$ with $\|w\|_{L^q(\Omega)}=1$ and $(\mu,w)$ is an eigenpair of \eqref{meqn}.
\end{theorem}

\begin{theorem}\label{subopthm1}
Let $0<s<1$, $1<p<\infty$ and $1<q<p^{*}$. Suppose $\{u_n\}_{n\in\mathbb{N}}\subset X\cap L^q(\Omega)$ such that $\|u_n\|_{L^q(\Omega)}=1$ and $\lim_{n\to\infty}\|u_n\|_{X}^{p}=\lambda$.
Then there exists a subsequence $\{u_{n_j}\}_{j\in\mathbb{N}}$ which converges weakly in $X$ to $u\in X\cap L^q(\Omega)$ with $\|u\|_{L^q(\Omega)}=1$ such that  
$$
\lambda=\int_{\Om}|\nabla u|^p\,dx+\int_{\R^N}\int_{\R^N}\frac{|u(x)-u(y)|^{p}}{|x-y|^{N+ps}}\,dx dy.
$$
Moreover, $(\lambda,u)$ is an eigenpair of \eqref{meqn} and any associated eigenfunction of $\lambda$ are precisely the scalar multiple of those vectors at which $\lambda$ is reached.
\end{theorem}

Our final main result concerns the following qualitative properties of the eigenfunctions of \eqref{meqn}.
\begin{theorem}\label{regthm}
Let $0<s<1$, $1<p<\infty$ and $1<q<p^{*}$. Assume that $\lambda>0$ is an eigenvalue of the problem \eqref{meqn} and $u\in X\setminus\{0\}$ is a corresponding eigenfunction. Then $(a)$ $u\in L^\infty(\Omega)$. $(b)$ Moreover, if $u$ is nonnegative in $\Omega$, then $u>0$ in $\Omega$. Further, for every $\omega\Subset\Omega$ there exists a positive constant $c$ depending on $\omega$ such that $u\geq c>0$ in $\omega$.
\end{theorem}

\section{Preliminaries}
In this section, we establish some preliminary results that are crucial to prove our main results. To this end, we define the operator $A:X\to X^*$ by
\begin{equation}
\label{a}
\begin{split}
\langle Av,w\rangle=\\ \int_{\Omega}|\nabla v|^{p-2}\nabla v\nabla w\,dx+\int_{\mathbb{R}^N}\int_{\R^N}\frac{|v(x)-v(y)|^{p-2}(v(x)-v(y))(w(x)-w(y))}{|x-y|^{N+ps}}\,dx dy,
\end{split}
\end{equation}
and $B:L^q(\Omega)\to (L^q(\Omega))^*$ by
\begin{equation}\label{b}
\begin{split}
\langle B(v),w\rangle=\int_{\Omega}|v|^{q-2}vw\,dx.
\end{split}
\end{equation}
The symbols $X^*$ and $(L^q(\Omega))^*$ denotes the dual of $X$ and $L^q(\Omega)$ respectively. First, we have the following result.
\begin{lemma}\label{newlem}
$(i)$ The operators $A$ defined by \eqref{a} and $B$ defined by \eqref{b} are continuous. $(ii)$ Moreover, $A$ is bounded, coercive and monotone.
\end{lemma}
\begin{proof}
\noindent
$(i)$ \textbf{Continuity:} Suppose $v_n\in X$ such that $v_n\to v$ in the norm of $X$. Thus, up to a subsequence $\nabla v_{n}\to \nabla v$ in $\Om$. We observe that 
\begin{equation}\label{mfd}
\||\nabla v_{n}|^{p-2}\nabla v_{n}\|_{L^\frac{p}{p-1}(\Om)}\leq \|\nabla v_{n}\|^{p-1}_{L^p(\Omega)}\leq c, 
\end{equation}
for some constant $c>0$, which is independent of $n$. Thus, up to a subsequence, we have
\begin{equation}\label{fc}
|\nabla v_{n}|^{p-2}\nabla v_{n}\to |\nabla v|^{p-2}\nabla v\text{ weakly in }L^{p'}(\Om).
\end{equation}
Moreover, up to a subsequence, we have
\begin{equation}\label{mffc}
\frac{|v_n(x)-v_n(y)|^{p-2}(v_n(x)-v_n(y))}{|x-y|^\frac{N+ps}{p'}}\to\frac{|v(x)-v(y)|^{p-2}(v(x)-v(y))}{|x-y|^\frac{N+ps}{p'}} 
\end{equation}
weakly in $L^{p'}(\R^{2N})$. Since, the weak limit is independent of the choice of the subsequence, as a consequence of \eqref{fc} and \eqref{mffc}, we have 
$$
\lim_{n\to\infty}\langle Av_n,w\rangle=\langle Av,w\rangle
$$
for every $w\in X$. Thus $A$ is continuous. Similarly, we obtain $B$ is continuous.
\vskip 0.2cm
\noindent
$(ii)$ \textbf{Boundedness:}
Using Cauchy-Schwartz and H\"older's inequality, we observe that
\begin{equation}\label{mest}
\begin{split}
\langle Av,w\rangle&=\int_{\Om}|\nabla v|^{p-2}\nabla v\nabla w\,dx\\
&\quad\quad+\int_{\R^N}\int_{\R^N}\frac{|v(x)-v(y)|^{p-2}(v(x)-v(y))(w(x)-w(y))}{|x-y|^{N+ps}}\,dx dy\\
&\leq\int_{\Om}|\nabla v|^{p-1}|\nabla w|\,dx+\int_{\R^N}\int_{\R^N}\frac{|v(x)-v(y)|^{p-1}|w(x)-w(y)|}{|x-y|^{N+ps}}\,dx dy\\
&\leq\Big(\int_{\Om}|\nabla v|^p\,dx\Big)^\frac{p-1}{p}\Big(\int_{\Om}|\nabla w|^p\,dx\Big)^\frac{1}{p}\\
&\quad\quad+\Big(\int_{\R^N}\int_{\R^N}\frac{|v(x)-v(y)|^p}{|x-y|^{N+ps}}\,dx dy\Big)^\frac{p-1}{p}\Big(\int_{\R^N}\int_{\R^N}\frac{|w(x)-w(y)|^p}{|x-y|^{N+ps}}\,dx dy\Big)^\frac{1}{p}\\
&\leq\Bigg[\Big(\int_{\Om}|\nabla v|^p\,dx\Big)^\frac{p-1}{p}+\Big(\int_{\R^N}\int_{\R^N}\frac{|v(x)-v(y)|^p}{|x-y|^{N+ps}}\,dx dy\Big)^\frac{p-1}{p}\Bigg]\|w\|_{X}\\
&\leq \Big(\int_{\Om}|\nabla v|^p\,dx+\int_{\R^N}\int_{\R^N}\frac{|v(x)-v(y)|^p}{|x-y|^{N+ps}}\,dx dy\Big)^\frac{p-1}{p}\|w\|_{X}=\|v\|_{X}^{p-1}\|w\|_{X},
\end{split}
\end{equation}
Therefore, we have
$$
\|Av\|_{X^*}=\sup_{\|w\|_{X}\leq 1}|\langle Av,w\rangle|\leq\|v\|_{X}^{p-1}\|w\|_{X}\leq\|v\|^{p-1}_{X}.
$$
Thus, $A$ is bounded.

\noindent
\textbf{Coercivity:}  We observe that 
$$
\langle Av,v\rangle=\int_{\Om}|\nabla v|^p\,dx+\int_{\R^N}\int_{\R^N}\frac{|v(x)-v(y)|^p}{|x-y|^{N+ps}}\,dx dy=\|v\|_{X}^p.
$$
Since $p>1$, we have $A$ is coercive.\\

\noindent
\textbf{Monotonicity:} For $u\in X$, let us denote by
$$
\mathcal{A}(u(x,y))=|u(x)-u(y)|^{p-2}(u(x)-u(y)),\quad d\mu=\frac{dx dy}{|x-y|^{N+ps}}.
$$
Recall the following algebraic inequality from \cite[Lemma $2.1$]{Dama}: there exists a constant $C=C(p)>0$ such that
\begin{equation}\label{algineq}
\langle |a|^{p-2}a-|b|^{p-2}b, a-b \rangle\geq C(p)(|a|+|b|)^{p-2}|a-b|^2, \,\,1<p<\infty,
\end{equation}
for any $a,b\in\mathbb{R}^N$.

Thus, for every $v,w\in X$, we have
\begin{equation*}
\begin{split}
&\langle Av-Aw,v-w\rangle\\
&=\int_{\Om}\langle|\nabla v|^{p-2}\nabla v-|\nabla w|^{p-2}\nabla w,\nabla (v-w)\rangle\,dx\\
&\quad+\int_{\R^N}\int_{\R^N}\Big(\mathcal{A}(v(x,y))-\mathcal{A}(w(x,y))\Big)((v(x)-w(x))-(v(y)-w(y)))\,d\mu\\
&\geq 0.
\end{split}
\end{equation*}
Thus, $A$ is monotone.\\
\end{proof}

\begin{lemma}\label{auxlmab}
The operators $A$ defined by \eqref{a} and $B$ defined by \eqref{b} satisfy the following properties: 
\vskip 0.2cm

\noindent
$(H_1)$ $A(tv)=|t|^{p-2}tA(v)\quad\forall t\in\mathbb{R}\quad \text{and}\quad\forall v\in X$.
\vskip 0.2cm

\noindent
$(H_2)$ $B(tv)=|t|^{q-2}tB(v)\quad\forall t\in\mathbb{R}\quad \text{and}\quad\forall v\in L^q(\Omega)$.
\vskip 0.2cm

\noindent
$(H_3)$ $\langle A(v),w\rangle\leq\|v\|_{X}^{p-1}\|w\|_{X}$ for all $v,w\in X$, where the equality holds if and only if $v=0$ or $w=0$ or $v=t w$ for some $t>0$.
\vskip 0.2cm

\noindent
$(H_4)$ $\langle B(v),w\rangle\leq\|v\|_{L^q(\Omega)}^{q-1}\|w\|_{L^q(\Omega)}$ for all $v,w\in {L^q(\Omega)}$, where the equality holds if and only if $v=0$ or $w=0$ or $v=t w$ for some $t\geq 0$.
\vskip 0.2cm

\noindent
$(H_5)$ For every $w\in L^q(\Omega)\setminus\{0\}$ there exists $u\in X\setminus\{0\}$ such that
$$
\langle A(u),v\rangle=\langle B(w),v\rangle\quad\forall\quad v\in X.
$$
\end{lemma}
\begin{proof}
\vskip 0.2cm

\noindent
$(H_1)$ Follows by the definition of $A$.
\vskip 0.2cm

\noindent
$(H_2)$  Follows by the definition of $B$.
\vskip 0.2cm

\noindent
$(H_3)$ First, we note that from \eqref{mest} the inequality
$\langle A(v),w\rangle\leq\|v\|_{X}^{p-1}\|w\|_{X}$ holds for all $v,w\in X$. Let the equality
\begin{equation}\label{mequal}
\langle A(v),w\rangle=\|v\|_{X}^{p-1}\|w\|_{X}
\end{equation}
holds for every $v,w\in X$. We claim that either $v=0$ or $w=0$ or $v=tw$ for some constant $t>0$. Indeed, if $v=0$ or $w=0$, then \eqref{mequal} is true. Therefore, we assume that both $v$ and $w$ are not identically zero in $\Om$ and prove that $v=dw$ for some constant $d>0$. By the estimate \eqref{mest} if the equality \eqref{mequal} holds, then we have
\begin{equation}\label{mequal1}
\langle A(v), w\rangle=\int_{\Om}|\nabla v|^{p-1}|\nabla w|\,dx+\int_{\R^N}\int_{\R^N}\frac{|v(x)-v(y)|^{p-1}|w(x)-w(y)|}{|x-y|^{N+ps}}\,dx dy,
\end{equation}
which gives us
\begin{equation}\label{mCS}
\int_{\Om}f(x)\,dx+\int_{\R^N}\int_{\R^N}\frac{g(x,y)\,dx dy}{|x-y|^{N+ps}}=0,
\end{equation}
where 
$$
f(x)=|\nabla v|^{p-1}|\nabla w|-|\nabla v|^{p-2}\nabla v\nabla w
$$
and
$$
g(x,y)=|v(x)-v(y)|^{p-1}|w(x)-w(y)|-|v(x)-v(y)|^{p-2}(v(x)-v(y))w(x)-w(y).
$$
By Cauchy-Schwartz inequality, we have $f\geq 0$ in $\Om$ and $g\geq 0$ in $\R^N\times\R^N$. Hence using these facts in \eqref{mCS}, we have $f=0$ in $\Om$, which reduces to
\begin{equation}\label{mfCS}
|\nabla v|^{p-1}|\nabla w|=|\nabla v|^{p-2}\nabla v\nabla w\text{ in }\Om.
\end{equation}
Hence $\nabla v(x)=c(x)\nabla w(x)$ for some $c(x)\geq 0$.
On the other hand, if the equality \eqref{mequal} holds, then by the estimate \eqref{mest} we have
\begin{equation}\label{mequal2}
\begin{split}
f_1-f_2=g_2-g_1,
\end{split}
\end{equation}
where
$$
f_1=\int_{\Om}|\nabla v|^{p-1}|\nabla w|\,dx,\quad f_2=\Big(\int_{\Om}|\nabla v|^p\,dx\Big)^\frac{p-1}{p}\Big(\int_{\Om}|\nabla w|^p\,dx\Big)^\frac{1}{p},
$$
$$
g_1=\int_{\R^N}\int_{\R^N}\frac{|v(x)-v(y)|^{p-2}(v(x)-v(y))(w(x)-w(y))}{|x-y|^{N+ps}}\,dx dy
$$
and
$$
g_2=\Big(\int_{\R^N}\int_{\R^N}\frac{|v(x)-v(y)|^p}{|x-y|^{N+ps}}\,dx dy\Big)^\frac{p-1}{p}\Big(\int_{\R^N}\int_{\R^N}\frac{|w(x)-w(y)|^p}{|x-y|^{N+ps}}\,dx dy\Big)^\frac{1}{p}.
$$
By the H\"older's inequality, we know that $f_1-f_2\leq 0$ and $g_2-g_1\geq 0$. Therefore, we obtain from \eqref{mequal2} that
$$
f_1=f_2\text{ and }g_1=g_2.
$$
Since $f_1=f_2$, the equality in H\"older's inequality holds, which gives
\begin{equation}\label{mCSeq2}
|\nabla v|=d|\nabla w|\text{ in }\Om,
\end{equation}
for some constant $d>0$. Therefore, $c(x)=d$. Therefore, $v=d\,w$ a.e. in $\Om$ for some constant $d>0$. Hence, the property $(H3)$ is verified.

\vskip 0.2cm

\noindent
$(H_4)$ This property can be verified similarly as in $(H_3)$.
\vskip 0.2cm

\noindent
$(H_5)$ Note that by Lemma \ref{Xuthm}, it follows that $X$ is a separable and reflexive Banach space. By Lemma \ref{newlem}, the operator $A:X\to X^*$ is bounded, continuous, coercive and monotone.

By the Sobolev embedding theorem, we have $X$ is continuously embedded in $L^q(\Omega)$. Therefore, $B(w)\in X^*$ for every $w\in L^q(\Omega)\setminus\{0\}$.

Hence, by Theorem \ref{MB}, for every $w\in L^q(\Omega)\setminus\{0\}$, there exists $u\in X\setminus\{0\}$ such that
$$
\langle A(u),v\rangle=\langle B(w),v\rangle\quad\forall v\in X.
$$
Hence the property $(H_5)$ holds. This completes the proof.
\end{proof}

\section{Proof of the main results:}
\vskip 0.2cm
\noindent
\textbf{Proof of Theorem \ref{newthm}:}
\vskip 0.2cm
\noindent
$(a)$ First we recall the definition of the operators $A:X\to X^*$ from \eqref{a} and $B:L^q(\Omega)\to (L^{q}(\Omega))^*$ from \eqref{b} respectively. Then, noting the property $(H_5)$ from Lemma \ref{auxlmab} and proceeding along the lines of the proof in \cite[page $579$ and pages $584-585$]{Ercole}, the result follows.
\vskip 0.2cm
\noindent
$(b)$ Note that by Lemma \ref{Xuthm}, $X$ is uniformly convex Banach space and by the Sobolev embedding theorem, $X$ is compactly embedded in $L^q(\Omega)$. Next, using Lemma \ref{newlem}-$(i)$, the operators $A:X\to X^*$ and $B:L^q(\Omega)\to (L^q(\Omega))^*$ are continuous and by Lemma \ref{auxlmab}, the properties $(H_1)-(H_5)$ holds. Noting these facts, the result follows from \cite[page $579$, Theorem 1]{Ercole}. \qed
\vskip 0.2cm
\noindent
\textbf{Proof of Theorem \ref{subopthm1}:} The proof follows due to the same reasoning as in the proof of Theorem \ref{newthm}-$(b)$ except that here we apply \cite[page $583$, Proposition $2$]{Ercole} in place of \cite[page $579$, Theorem 1]{Ercole}. 
\vskip 0.2cm
\noindent
\textbf{Proof of Theorem \ref{regthm}:}
\vskip 0.2cm
\noindent
$(a)$ Due to the homogeneity of the equation \eqref{meqn}, without loss of generality, we assume that $\|u\|_{L^q(\Omega)}=1$. Let $k\geq 1$ and set $L(k):=\{x\in\Omega:u(x)>k\}$. Choosing $v=(u-k)^+$ as a test function in \eqref{mwksol}, we obtain
\begin{equation}\label{regtst1}
\begin{split}
&\int_{L(k)}|\nabla u|^p\,dx+\int_{\mathbb{R}^N}\int_{\mathbb{R}^N}\frac{|u(x)-u(y)|^{p-2}(u(x)-u(y))(v(x)-v(y))}{|x-y|^{N+ps}}\,dx dy\,\\
&=\lambda\int_{L(k)}u^{q-1}(u-k)\,dx\leq\lambda\int_{L(k)}u^{q-1}(u-k)\,dx.
\end{split}
\end{equation}
From the estimate $(3.9)$ on page 10 in \cite{GU22}, it follows that the second integral in the left hand side of \eqref{regtst1} is nonnegative. Therefore, we have
\begin{equation}\label{regtst2}
\int_{L(k)}|\nabla u|^p\,dx\leq\lambda\int_{L(k)}u^{q-1}(u-k)\,dx\leq\lambda\int_{L(k)}u^{q-1}(u-k)\,dx.
\end{equation}
The rest of the proof follows by arguing along the lines of the proof of \cite[Theorem 3.4, Pages 7--8]{GUsub}. For convenience of the reader, we prove it here.\\
\textbf{Case $I$.} Let $q\leq p$, then since $k\geq 1$, over the set $L(k)$, we have $u^{q-1}\leq u^{p-1}$. Therefore, from \eqref{regtst1} we have
\begin{equation}\label{regtst22}
\begin{split}
\int_{L(k)}|\nabla u|^p\,dx&\leq\lambda\int_{L(k)}|u|^{p-1}(u-k)\,dx\\
&\leq\lambda\int_{L(k)}(2^{p-1}(u-k)^{p}+2^{p-1}k^{p-1}(u-k))\,dx,
\end{split}
\end{equation}
where to obtain the last inequality above, we have used the inequality $(a+b)^{p-1}\leq 2^{p-1}(a^{p-1}+b^{p-1})$ for $a,b\geq 0$. Using the Sobolev inequality \eqref{embeqn} with $r=p$ in \eqref{regtst2} we obtain
\begin{equation}\label{regtst3}
\begin{split}
(1-S\lambda 2^{p-1}|L(k)|^\frac{p}{\nu})\int_{L(k)}(u-k)^p\,dx&\leq\lambda S 2^{p-1}k^{p-1}|L(k)|^\frac{p}{\nu}\int_{L(k)}(u-k)\,dx,
\end{split}
\end{equation}
where $S>0$ is the Sobolev constant. Note that $\|u\|_{L^1(\Omega)}\geq k|L(k)|$ and therefore for every $k\geq k_0=(2^p S\lambda)^\frac{\nu}{p}\|u\|_{L^1(\Omega)}$, we have $S\lambda2^{p-1}|L(k)|^\frac{p}{\nu}\leq\frac{1}{2}$. Using this fact in \eqref{regtst3}, for every
$k\geq\max\{k_0,1\}$, we get
\begin{equation}\label{regtst4}
\begin{split}
\int_{L(k)}(u-k)^p\,dx&\leq \lambda S 2^{p}k^{p-1}|L(k)|^\frac{p}{\nu}\int_{L(k)}(u-k)\,dx.
\end{split}
\end{equation}
Using H\"older's inequality and the estimate \eqref{regtst4} we obtain
\begin{equation}\label{regtst5}
\int_{L(k)}(u-k)\,dx\leq (\lambda S2^p)^\frac{1}{p-1}k|L(k)|^{1+\frac{p}{\nu(p-1)}}.
\end{equation}
Noting \eqref{regtst5}, by \cite[Lemma $5.1$]{Ural}, we get $u\in L^\infty(\Omega)$. \\
\textbf{Case $II.$} Let $q>p$, then using the inequality $(a+b)^{q-1}\leq 2^{q-1}(a^{q-1}+b^{q-1})$ for $a,b\geq 0$ in \eqref{regtst1} we get
\begin{equation}\label{regtstc2}
\int_{L(k)}|\nabla u|^p\,dx\leq\lambda\int_{L(k)}(2^{q-1}(u-k)^{q}+2^{q-1}k^{q-1}(u-k))\,dx.
\end{equation}
Now, using the Sobolev inequality \eqref{embeqn} with $r=q$ in the estimate \eqref{regtstc2} we obtain
\begin{equation}\label{regtstc21}
\begin{split}
\left(\int_{L(k)}(u-k)^q\,dx\right)^\frac{p}{q}&\leq S\lambda|L(k)|^{p(\frac{1}{q}-\frac{1}{p}+\frac{1}{\nu})}\int_{L(k)}(2^{q-1}(u-k)^q+2^{q-1}k^{q-1}(u-k))\,dx,
\end{split}
\end{equation}
where $S>0$ is the Sobolev constant. Since $\int_{L(k)}(u-k)^q\,dx\leq\|u\|^{q}_{L^q(\Omega)}=1$ and $q>p$, the quantity in the left side of \eqref{regtstc21} can be estimated from below as

\begin{equation}\label{regtstc22}
\left(\int_{L(k)}(u-k)^q\,dx\right)^\frac{p}{q}=\left(\int_{L(k)}(u-k)^q\,dx\right)^{\frac{p-q}{q}+1}\geq\int_{L(k)}(u-k)^q\,dx.
\end{equation}
Using \eqref{regtstc22} in \eqref{regtstc21} we get
\begin{equation}\label{regtstc23}
\begin{split}
&\Big(1-S\lambda 2^{q-1}|L(k)|^{p(\frac{1}{q}-\frac{1}{p}+\frac{1}{\nu})}\Big)\int_{L(k)}(u-k)^q\,dx\\
&\leq S\lambda2^{q-1}k^{q-1}|L(k)|^{p(\frac{1}{q}-\frac{1}{p}+\frac{1}{\nu})}\int_{L(k)}(u-k)\,dx.
\end{split}
\end{equation}
Let $\alpha={p(\frac{1}{q}-\frac{1}{p}+\frac{1}{\nu})}$, which is positive, since $1<q<p^{*}$. Choosing $k_1=(S\lambda2^q)^\frac{1}{\alpha}\|u\|_{L^1(\Omega)}$, due to the fact that $k|L(k)|\leq\|u\|_{L^1(\Omega)}$, we obtain for every $k\geq k_1$ that $S\lambda 2^{q-1}|L(k)|^\alpha\leq\frac{1}{2}$. Using this property in \eqref{regtstc23}, we have

\begin{equation}\label{regtstc24}
\begin{split}
\int_{L(k)}(u-k)^q\,dx&\leq S\lambda2^{q}k^{q-1}|L(k)|^\alpha\int_{L(k)}(u-k)\,dx.
\end{split}
\end{equation}
By H\"older's inequality and the estimate \eqref{regtstc24} we arrive at
\begin{equation}\label{regtstc25}
\int_{L(k)}(u-k)\,dx\leq (\lambda S2^q)^\frac{1}{q-1}k|L(k)|^{1+\frac{\alpha}{q-1}}.
\end{equation}
Noting \eqref{regtstc25}, by \cite[Lemma $5.1$]{Ural}, we get $u\in L^\infty(\Omega)$.
\vskip 0.2cm
\noindent
$(b)$ By \cite[Theorem $8.4$]{GK}, the result follows. \qed

\vskip 0.6cm

\noindent {{Prashanta Garain\\
Department of Mathematical Sciences,\\
Indian Institute of Science Education and Research Berhampur,\\
Berhampur, Odisha 760010, India\\
Department of Mathematics,\\
Indian Institute of Technology Indore,\\
Khandwa Road, Simrol, Indore 453552, India,\\ 
\textsf{e-mail}: pgarain92@gmail.com\\

\vskip 0.3cm

\noindent {Alexander Ukhlov\\
Department of Mathematics,\\
Ben-Gurion University of the Negev,\\ P.O.Box 653, Beer Sheva, 8410501, Israel\\
\textsf{e-mail}: ukhlov@math.bgu.ac.il\\

\end{document}